\documentclass[12pt]{amsart}

\usepackage[leqno]{amsmath}
\usepackage{amsthm}
\usepackage{amsfonts}
\usepackage{amssymb}
\usepackage[foot]{amsaddr}

\usepackage{hyperref}
\usepackage{stmaryrd}
\usepackage{eucal}
\usepackage{mathtools}
\usepackage[all]{xy}
\usepackage{tikz}
\usetikzlibrary{decorations.markings,intersections, shapes.geometric, arrows,matrix,positioning, patterns}
\usepackage{tikz-cd}
\usepackage{accents}
\newcommand*{\dt}[1]{%
   \accentset{\mbox{\large\bfseries .}}{#1}}
\CompileMatrices

\DeclareFontFamily{OT1}{rsfs}{}
\DeclareFontShape{OT1}{rsfs}{n}{it}{<-> rsfs10}{}
\DeclareMathAlphabet{\mathscr}{OT1}{rsfs}{n}{it}
\renewcommand*{\setminus}{-}

\DeclareMathOperator{\Bun}{Bun}

\DeclareMathOperator{\Gr}{Gr}
\DeclareMathOperator{\bGr}{\overline{\mathrm{Gr}}}

\DeclareMathOperator{\mW}{\mathcal W}
\DeclareMathOperator{\bW}{\overline{\mathcal W}}

\DeclareMathOperator{\Spec}{Spec}

\theoremstyle{plain}
  \newtheorem{theorem}[subsection]{Theorem}
  \newtheorem{proposition}[subsection]{Proposition}
  \newtheorem{lemma}[subsection]{Lemma}
  \newtheorem{corollary}[subsection]{Corollary}
  \newtheorem{conjecture}[subsection]{Conjecture} 
   
  \newtheorem{lem-def}[subsection]{Lemma-Definition}

\theoremstyle{definition}

\theoremstyle{remark}
  \newtheorem{example}[subsection]{Example}
  \newtheorem{remark}[subsection]{Remark}

\numberwithin{equation}{section}

   \DeclareMathOperator{\SL}{SL}

\begin{document}

\author{Yehao Zhou}
\address{Perimeter Institute for Theoretical Physics}
\email{yzhou3@pitp.ca}

\title{Note on some properties of generalized affine Grassmannian slices}

\begin{abstract}
    Braverman, Finkelberg and Nakajima introduced the generalized affine Grassmannian slices $\overline{\mathcal W}^{\lambda}_{\mu}$ and showed that they are Coulomb branches of $3d$ $\mathcal N=4$ gauge theories. We prove a conjecture of theirs that a transversal slice of $\mathcal W^{\nu}_{\mu}$ in $\overline{\mathcal W}^{\lambda}_{\mu}$ is isomorphic to $\overline{\mathcal W}^{\lambda}_{\nu}$. In addition, they conjecture that Coulomb branches of $3d$ $\mathcal N=4$ gauge theories have symplectic singularities, and we confirm this conjecture for generalized affine Grassmannian slices $\overline{\mathcal W}^{\lambda}_{\mu}$. Along the way we give a new proof of the smoothness of $\mathcal W^{\nu}_{\mu}$, which has been previously proven by Muthiah and Weekes using different method.
\end{abstract}

\date{}

\maketitle

\section{Introduction}
For a complex reductive group $G$, and a pair of dominant coweights $\lambda,\mu$ such that $\mu\le\lambda$, the affine Grassmannian slice $\bW^{\lambda}_{\mu}$ is a transversal slice of $G_{\mathcal O}$-orbit $\Gr^{\mu}$ in the $G_{\mathcal O}$-orbit closure $\bGr^{\lambda}$. As they are transversal slices, they contain important information on singularities of $\bGr^{\lambda}$. It is shown that $\bW^{\lambda}_{\mu}$ has symplectic singularities \cite{kamnitzer2014yangians}, and this implies that $\bGr^{\lambda}$ has rational and Gorenstein singularities. The also play significant role in representation theory, for example the intersection cohomology of $\bW^{\lambda}_{\mu}$ can be identified with the the weight $\mu$ subspace $V^{\lambda}(\mu)$ of the irreducible representation $V^{\lambda}$ of the Langlands dual group $G^{\vee}$ of highest weight $\lambda$, by the geometric Satake correspondence. 

Recently, Braverman, Finkelberg and Nakajima defined a generalization of affine Grassmannian slices in their seminal paper \cite{braverman2016coulomb}, which is still denoted by $\bW^{\lambda}_{\mu}$ but $\mu$ is not required to be dominant (we will recall its definition in Section \ref{Sec: Smoothness}). They have shown that the generalized affine Grassmannian slices are Coulomb branches of $3d$ $\mathcal N=4$ quiver gauge theories, when $G$ is of ADE-type, and later work \cite{nakajima2019coulomb} extended this result to BCFG-type using quivers with symmetrizers.

Physically, a Coulomb branch is an irreducible component of the moduli space of vacua of a supersymmetric gauge theory \cite{bullimore2015coulomb}. It is argued in the physics literature that Coulomb branch of $3d$ $\mathcal N=4$ gauge theory is the moduli space of singular $G_c$-monopoles on $\mathbb R^3$ \cite{hanany1997type,tong1999three,de1997mirror,de1997mirror2,cherkis1998singular}, denoted by $\mathcal M_C(\mu,\lambda)$, where $G_c$ is the compact form of $G$, $\mu,\lambda:S^1\to G_c$ are coweights attached at $\infty ,0$ of $\mathbb R^3$. Here $\mu$ is called the magnetic charge, and $\lambda$ is called the monopole charge which is assumed to be dominant.

The variety $\mathcal M_C(\mu,\lambda)$ is a remarkable object. It carries a natural Poisson structure, together with a canonical quantization $\mathcal M^q_C(\mu,\lambda)$ with respect to the natural Poisson structure \cite{braverman2016towards}. It is expected to be symplectic dual to the Higgs branch $\mathcal M_H(\mu,\lambda)$ of the same gauge theory \cite{braden2012quantizations,braden2014quantizations}, and the symplectic duality has been worked out in many cases when the gauge group is simply-laced \cite{webster2016koszul,kamnitzer2018category}. These  features make it an ideal test ground for geometric representation theory.

In this note we will focus on some basic geometric properties of $\mathcal M_C(\mu,\lambda)$. According to the ``monopole bubbling" \cite{ito2012line,gomis2012exact}, the moduli space of singular $G_c$-monopoles on $\mathbb R^3$ has a decomposition
$$\mathcal M_C(\mu,\lambda)=\coprod_{\substack{\nu:\text{ dominant}\\ \mu\le \nu\le \lambda }}\mathring{\mathcal M}_C(\mu,\nu)$$
where $\mathring{\mathcal M}_C(\mu,\nu)$ denotes the moduli space of genuine $G_c$-monopoles on $\mathbb R^3$ with magnetic charge $\mu$ and monopole charge $\nu$. In \cite{nakajima2015questions}, Nakajima made the following conjecture:
\begin{conjecture}\label{Conj: Main}
$\mathring{\mathcal M}_C(\mu,\nu)$ is a symplectic leaf of $\mathcal M_C(\mu,\lambda)$, and its transversal slice in $\mathcal M_C(\mu,\lambda)$ is isomorphic to $\mathcal M_C(\nu,\mu)$.
\end{conjecture}
 
By the results in \cite{braverman2016coulomb,nakajima2019coulomb}, generalized affine Grassmannian slices are examples of Coulomb branches of $3d$ $\mathcal N=4$ quiver gauge theories (with appropriate symmetrizers if needed), and in this case we have $\mathcal M_C(\mu,\lambda)=\bW^{\lambda}_{\mu}$ and $\mathring{\mathcal M}_C(\mu,\nu)=\mW^{\nu}_{\mu}$, in particular we have the following ``monopole bubbling" decomposition
$$\overline {\mathcal W}^{\lambda}_{\mu}=\coprod_{\substack{\nu:\text{ dominant}\\ \mu\le \nu\le \lambda }}\mathcal{W}^{\nu}_{\mu}.$$
The main result of this note is that we confirm the Conjecture \ref{Conj: Main} for $\overline {\mathcal W}^{\lambda}_{\mu}$. We will review the definitions of $\overline {\mathcal W}^{\lambda}_{\mu}$ and $\mathcal{W}^{\nu}_{\mu}$, and explain the nature of this decomposition  in Section \ref{Sec: Smoothness}. It is explained in \cite[Remark 3.19]{braverman2016coulomb} that to prove the first part of Conjecture \ref{Conj: Main} for $\overline {\mathcal W}^{\lambda}_{\mu}$, it suffices to prove the following
\begin{theorem}\label{Thm: Smoothness}
$\mathcal{W}^{\nu}_{\mu}$ is smooth.
\end{theorem}
This theorem has been proven in \cite{muthiah2019symplectic}. In Section \ref{Sec: Smoothness}, we will use a different method to reprove Theorem \ref{Thm: Smoothness}, which is inspired from the proof of the Cohen-Macaulay property of $\overline {\mathcal W}^{\lambda}_{\mu}$ in \cite[Lemma 2.16]{braverman2016coulomb}. In fact, we will prove a slightly stronger result, which holds for the Beilinson-Drinfeld version of the generalized affine Grassmannian slices, see Proposition \ref{Prop: Smooth Locus}. In Section \ref{Sec: Symplectic Singularities}, we prove the second part of Conjecture \ref{Conj: Main} for $\overline {\mathcal W}^{\lambda}_{\mu}$:
\begin{theorem}\label{Thm: Transversal Slice}
A transversal slice of $\mathcal W^{\nu}_{\mu}$ in $\overline{\mathcal W}^{\lambda}_{\mu}$ is isomorphic to $\overline{\mathcal W}^{\lambda}_{\nu}$.
\end{theorem}
\noindent Note that this theorem is known in \cite{braverman2016coulomb} for two special cases: dominant $\mu$ and $\mu\le w_0\lambda$.\\

In \cite{braverman2016towards}, Braverman \textit{et. al.} made the following conjecture on the singularities of the Coulomb branches of $3d$ $\mathcal N=4$ gauge theories.
\begin{conjecture}\label{Conj: Symplectic Singularities}
$\mathcal M_C(\mu,\lambda)$ has symplectic singularities.
\end{conjecture}
For gauge theories defined by simple quivers, this conjecture has been proven by Weekes using Coulomb branch techniques \cite{weekes2020quiver}. In particular, this implies that $\overline {\mathcal W}^{\lambda}_{\mu}$ has symplectic singularities when $G$ is of ADE-type. Using techniques developed in \cite{nakajima2019coulomb}, his result extends to BCFG-type as well. We reprove the special case of this conjecture for $\overline {\mathcal W}^{\lambda}_{\mu}$ in Section \ref{Sec: Symplectic Singularities} without Coulomb branch machinery. Finally, we end this note with generalization of \cite[Theorem 2.7, 2.9]{kamnitzer2014yangians} to not necessarily dominant $\mu$.

\subsection*{Acknowledgement} The author would like to thank Alex Weekes and Vasily Krylov for discussions and comments on earlier version of this note. This research is supported by Perimeter Institute for Theoretical Physics. Research at Perimeter Institute is supported by the Government of Canada through Industry Canada and by the Province of Ontario through the Ministry of Research and Innovation.

\section{Smoothness of {$\mathcal{W}^{\lambda}_{\mu}$}}\label{Sec: Smoothness}

Let $G$ be a split reductive group scheme over $\mathbb{Z}$, let's recall the definition of the generalized affine Grassmannian slice $\overline {\mathcal W}^{\lambda}_{\mu}$ for a dominant coweight $\lambda$, and a coweight $\mu$ such that $\lambda\ge \mu$ \cite{braverman2016coulomb}. First, we have a functor $\overline {\mathsf W}^{{\lambda}}_{\mu}$ sending a scheme $S$ to the data
\begin{itemize}
\item[(1)] a $G$-bundle $\mathcal F_G$ on $\mathbb{P}^1_S$;
\item[(2)] a trivialization $\sigma$ of $\mathcal F_G$ on $\mathbb{P}^1_S\setminus 0$ with pole of order $\le\lambda$;
\item[(3)] a $B$-structure $\phi$ on $\mathcal F_G$ of degree $w_0\mu$ having fiber $B\times S$ at $\infty\times S\subset \mathbb{P}^1_S$.
\end{itemize}
It's easy to see that $\overline {\mathsf W}^{{\lambda}}_{\mu}$ is represented by a finite type scheme over $\mathbb{Z}$, and we define $\overline {\mathcal W}^{\lambda}_{\mu}$ to be the reduced scheme $\overline {\mathsf W}^{{\lambda}}_{\mu,\mathrm{red}}$.

From the definition, we read out that there is a Cartesian diagram
\begin{center}
\begin{tikzcd}
\overline {\mathcal W}^{{\lambda}}_{\mu}  \arrow[d,"\mathbf p"]  \arrow[r,"\mathbf r"] & \Bun_B^{w_0\mu}(\mathbb{P}^1,\infty) \arrow[d,"p"]\\
\bGr^{\lambda} \arrow[r,"h^{\rightarrow}_G"]  & \Bun_G(\mathbb P^1,\infty)
\end{tikzcd}
\end{center}
where $\Bun_G(\mathbb{P}^1,\infty)$ is the stack of $G$-bundles on $\mathbb{P}^1$ with a trivialization at $\infty$, and $\Bun_B^{w_0\mu}(\mathbb{P}^1,\infty)$ is the degree $w_0\mu$ component. In this diagram, the map $p$ sends a $B$-bundle to its induced $G$-bundle, and the map $h^{\rightarrow}_G$ sends a pair $(\mathcal F_G,\sigma)$ in $\bGr^{\lambda}$ to the $G$-bundle $\mathcal F_G$ together with the canonical trivialization induced from $\sigma$ at $\infty$.

Recall that there is a decomposition $\bGr^{\lambda}=\coprod_{\nu\le \lambda}\Gr^{\nu}$ for dominant $\nu$, thus we have corresponding decomposition for generalized affine Grassmannian slice:
$$\overline {\mathcal W}^{\lambda}_{\mu}=\coprod_{\substack{\nu:\text{ dominant}\\ \mu\le \nu\le \lambda }}\mathcal{W}^{\nu}_{\mu}.$$
This is known as the ``monopole bubbling" in the physics literature.

\subsection{Beilinson-Drinfeld generalized affine Grassmannian slice} If moreover there is a decomposition $\lambda=\sum_{i=1}^N \lambda_i$, we can define the Beilinson-Drinfeld version of the generalized affine Grassmannian slice $\mathcal{W}^{\underline{\lambda}}_{\mu}$, where $\underline{\lambda}=(\lambda_1,\cdots,\lambda_N)$ is the defect type. First, we define a functor $\mathsf W^{\underline{\lambda}}_{\mu}$ sending a scheme $S$ to the data
\begin{itemize}
\item[(1)] a collection of $S$-points $\underline{z}=(z_1,\cdots,z_N)\in \mathbb{A}^N(S)$;
\item[(2)] a $G$-bundle $\mathcal F_G$ on $\mathbb{P}^1_S$;
\item[(3)] a trivialization $\sigma$ of $\mathcal F_G$ on $\mathbb{P}^1_S\setminus \{z_1,\cdots,z_N\}$ with pole \textbf{exactly} of order $\sum_{i=1}^N\lambda_i\cdot z_i$;
\item[(4)] a $B$-structure $\phi$ on $\mathcal F_G$ of degree $w_0\mu$ having fiber $B\times S$ at $\infty\times S\subset \mathbb{P}^1_S$.
\end{itemize}
It's easy to see that $\mathsf W^{\underline{\lambda}}_{\mu}$ is represented by a finite type scheme over $\mathbb{Z}$, and then we define $\mathcal W^{\underline{\lambda}}_{\mu}=\mathsf W^{\underline{\lambda}}_{\mu,\mathrm{red}}$.

\begin{proposition}\label{Prop: Smoothness}
The natural morphism $\mathcal W^{\underline{\lambda}}_{\mu}\to \mathbb{A}^N$ is smooth.
\end{proposition}

\begin{proof}
We can read out from the definition that $\mathsf W^{\underline{\lambda}}_{\mu}$ fits into the commutative diagram
\begin{center}
\begin{tikzcd}
\Spec \mathbb{Z} \arrow[d,"i"] & \mathsf {Gr}^{\underline{\lambda}}_{\mathrm{BD}} \arrow[l]  \arrow[d] & \mathsf W^{\underline{\lambda}}_{\mu}  \arrow[l]  \arrow[d,"j"]\\
\Bun_G(\mathbb{P}^1,\infty)  & \mathsf H^{\underline{\lambda}}_{\mathrm{BD}} \arrow[l,"h^{\leftarrow}_G" ']  \arrow[d,"h^{\rightarrow}_G"] & \mathsf {C}^{\underline{\lambda}}_{\mu}  \arrow[l,"p'" ']  \arrow[d,"q"]\\
&\Bun_G(\mathbb{P}^1,\infty) &\Bun_B^{w_0\mu}(\mathbb{P}^1,\infty) \arrow[l,"p" '] 
\end{tikzcd}
\end{center}
where all squares are Cartesian. In this diagram, $\Bun_G(\mathbb{P}^1,\infty)  \overset{h^{\leftarrow}_G}{\longleftarrow} \mathsf H^{\underline{\lambda}}_{\mathrm{BD}}\overset{h^{\rightarrow}_G}{\longrightarrow} \Bun_G(\mathbb{P}^1,\infty)$ is the Beilinson-Drinfeld Hecke stack, defined by sending a scheme $S$ to the groupoid of
\begin{itemize}
\item[(1)] a collection of $S$-points $\underline{z}=(z_1,\cdots,z_N)\in \mathbb{A}^N(S)$;
\item[(2)] a $G$-bundle $\mathcal F^0_G$ on $\mathbb{P}^1_S$ with a trivialization $\sigma^0_{\infty}$ at $\infty$;
\item[(3)] a $G$-bundle $\mathcal F^1_G$ on $\mathbb{P}^1_S$ with a trivialization $\sigma^1_{\infty}$ at $\infty$;
\item[(4)]  an isomorphism $\varphi:\mathcal F^0_G|_{\mathbb{P}^1_S\setminus \{z_1,\cdots,z_N\}}\cong \mathcal F^1_G|_{\mathbb{P}^1_S\setminus \{z_1,\cdots,z_N\}}$ with pole \textbf{exactly} of order $\sum_{i=1}^N\lambda_i\cdot z_i$, and $\varphi$ maps $\sigma^0_{\infty}$ to $\sigma^1_{\infty}$;
\end{itemize}
$h^{\leftarrow}_G$ (resp. $h^{\rightarrow}_G$) maps the above data to $(\mathcal F^0_G,\sigma^0_{\infty})$ (resp. $(\mathcal F^1_G,\sigma^1_{\infty})$). $i:\Spec \mathbb{Z}\to \Bun_G(\mathbb{P}^1,\infty)$ is the trivial $G$-bundle with the canonical trivialization at $\infty$. Note that $\mathsf {C}^{\underline{\lambda}}_{\mu}$ is nothing but a variation of $\mathsf W^{\underline{\lambda}}_{\mu}$ by allowing the domain of the isomorphism $\sigma:\mathcal F^{\mathrm{triv}}_G|_{\mathbb{P}^1_S\setminus \{z_1,\cdots,z_N\}}\cong \mathcal F_G|_{\mathbb{P}^1_S\setminus \{z_1,\cdots,z_N\}}$ to vary.\\

Since $i$ induces an isomorphism between $\Spec \mathbb{Z}$ and the open substack $\Bun_G^{\mathrm{triv}}(\mathbb{P}^1,\infty)\subset \Bun_G(\mathbb{P}^1,\infty)$ of trivial $G$-bundles, we see that $j$ is an open immersion, thus the proposition follows from the following 
\begin{itemize}
\item[(•)] the natural morphism $\mathsf C^{\underline{\lambda}}_{\mu,\mathrm{red}}\to \mathbb{A}^N$ is smooth.
\end{itemize}
To prove this assertion, we notice that $h^{\rightarrow}_G$ is a locally trivial fibration in the smooth topology with fiber $\mathsf {Gr}^{\underline{\lambda}^*}_{\mathrm{BD}}$, hence $q$ has the same property. Since $\Bun_H(\mathbb{P}^1,\infty)$ is a smooth stack on $\mathbb{Z}$ for any smooth group scheme $H$, the assertion (•) follows from next lemma.
\end{proof}

\begin{lemma}
The natural morphism $\mathrm{Gr}^{\underline{\lambda}}_{\mathrm{BD}}:=\mathsf {Gr}^{\underline{\lambda}^*}_{\mathrm{BD},\mathrm{red}}\to \mathbb{A}^N$ is smooth.
\end{lemma}

\begin{proof}
Recall the well-known fact that $\overline{\mathrm{Gr}}^{\underline{\lambda}}_{\mathrm{BD}}\to \mathbb{A}^N$ is flat with geometrically reduced fibers. Now $\mathrm{Gr}^{\underline{\lambda}}_{\mathrm{BD}}\subset \overline{\mathrm{Gr}}^{\underline{\lambda}}_{\mathrm{BD}}$ and its geometric fibers along the projection to $\mathbb{A}^N$ are products of $G_{\mathcal{O}}$-orbits with reduced scheme structures, so the fibers are smooth, hence $\mathrm{Gr}^{\underline{\lambda}}_{\mathrm{BD}}\to \mathbb{A}^N$ is smooth.
\end{proof}

In fact, we can formulate a similar commutative diagram to the diagram that is used in the proof of Proposition \ref{Prop: Smoothness}, but allow the Hecke stack to have poles of order less than or equal to $\lambda$, then the up-right corner is $\overline{\mathcal W}^{\lambda}_{\mu}$, which maps smoothly to $\overline{\mathsf C}^{\lambda}_{\mu}$ (defined as the pullback of the right projection of the Hecke stack). Now the projection $\overline{\mathsf C}^{\lambda}_{\mu}\to \Bun^{w_0\mu}_B(\mathbb P^1, \infty)$ is a fibration with fiber $\bGr^{\lambda}$. It is known that the smooth locus of $\bGr^{\lambda}$ is exactly $\Gr^{\lambda}$, thus we have the following sharper result:
\begin{proposition}\label{Prop: Smooth Locus}
The natural morphism $\bW^{\underline{\lambda}}_{\mu}\to \mathbb{A}^N$ is flat and the smooth locus is $\mW^{\underline{\lambda}}_{\mu}$.
\end{proposition}

\section{Transversal Slice of $\mathcal W^{\nu}_{\mu}$ in $\overline{\mathcal W}^{\lambda}_{\mu}$}\label{Sec: Transversal Slice}

In this section we will prove Theorem \ref{Thm: Transversal Slice}. For simplicity we will assume that $G$ is a complex reductive group, but the proof holds for reductive groups over arbitrary algebraically closed field. We begin with proving a variant of this theorem.

\subsection{Infinite $\boldsymbol {\lambda}$} By sending $\lambda $ to infinity, we mean considering $\mW_{\mu}$ (resp. $\mW_{\nu}$) instead of $\overline{\mathcal W}^{\lambda}_{\mu}$ (resp. $\overline{\mathcal W}^{\lambda}_{\nu}$). First of all, we notice that there exists a natural $G_1[\![z^{-1}]\!]$ action on $\mathcal{W}_{\mu}$ which makes $\mathbf{p}: \mathcal{W}_{\mu}\to \mathrm{Gr}$ equivariant. In fact, we have \cite[2(xi)]{braverman2016coulomb}
$$\mathcal{W}_{\mu}=B_1[\![z^{-1}]\!]z^{\mu}B_1^{-}[\![z^{-1}]\!]=G_1[\![z^{-1}]\!]z^{\mu}B^{-}(\!(z^{-1})\!)/B^{-}[z]$$and the natural projection $\mathbf{p}: \mathcal{W}_{\mu}\to \mathrm{Gr}$ is a further step quotient by $G[z]$. So the left multiplication by $G_1[\![z^{-1}]\!]$ is an action on $\mathcal{W}_{\mu}$ and $\mathbf{p}$ is equivariant with respect to this action and the natural left multiplication of $G_1[\![z^{-1}]\!]$ on $\mathrm{Gr}$.

Next, recall that $\mathcal W_{\nu}$ is a transversal slice of $\mathrm{Gr}^{\nu}$ in $\mathrm{Gr}$, in fact there is an isomorphism: $$\mathcal W_{\nu}\cong G_1[\![z^{-1}]\!]\cap z^{\nu} G_1[\![z^{-1}]\!]z^{-\nu},\; U[z]z^{\nu}G[z]/G[z]\cong G[z]\cap z^{\nu} G_1[\![z^{-1}]\!]z^{-\nu}.$$Let's call the latter space $\mathcal U^{\nu}$, it's an open subvariety of $\mathrm {Gr}^{\nu}$. In this way we identify $\mathcal W_{\nu}$ and $\mathcal U^{\nu}$ with subgroups of $z^{\nu} G_1[\![z^{-1}]\!]z^{-\nu}$ and the multiplication maps
\begin{align*}
m_l:G_1[\![z^{-1}]\!]\cap z^{\nu} G_1[\![z^{-1}]\!]z^{-\nu}\times G[z]\cap z^{\nu} G_1[\![z^{-1}]\!]z^{-\nu}\to z^{\nu} G_1[\![z^{-1}]\!]z^{-\nu}\\
m_r:G[z]\cap z^{\nu} G_1[\![z^{-1}]\!]z^{-\nu}\times G_1[\![z^{-1}]\!]\cap z^{\nu} G_1[\![z^{-1}]\!]z^{-\nu}\to z^{\nu} G_1[\![z^{-1}]\!]z^{-\nu}
\end{align*}
are isomorphisms of schemes. We focus on $m_l$ for now. Since $z^{\nu} G_1[\![z^{-1}]\!]z^{-\nu}\to z^{\nu} G_1[\![z^{-1}]\!]z^{-\nu}\cdot z^{\nu}G[z]/G[z]$ is open immersion, we obtain an open immersion $$m_l:\mathcal W_{\nu}\times \mathcal U^{\nu} \to \mathrm{Gr},$$such that $1\times \mathcal U^{\nu}$ (resp. $ \mathcal W_{\nu} \times 1$) are mapped identically to $\mathcal U^{\nu}$ (resp. $\mathcal W_{\nu}$).

The result we state below is a special case of the scheme-theoretic version of \cite[Corollary 3.2.21]{chriss2009representation}, we reprove it for completeness.

\begin{proposition}
A transversal slice of $\mathbf{p}^{-1}(\mathrm{Gr}^{\nu})=\mathcal W^{\nu}_{\mu}$ in $\mathcal W_{\mu}$ is isomorphic to $\mathcal W_{\nu}$.
\end{proposition}

\begin{proof}
Consider the Cartesian diagram
\begin{center}
\begin{tikzcd}
 X \arrow[r,"m_l'"] \arrow[d] & \mathcal W_{\mu} \arrow[d,"\mathbf{p}"]\\
\mathcal W_{\nu}\times \mathcal U^{\nu} \arrow[r,"m_l"]  & \mathrm{Gr}
\end{tikzcd}
\end{center}
so $m_l'$ is open immersion as well. We claim that $X\cong \mathcal W_{\nu} \times \mathbf{p}^{-1}(\mathcal U^{\nu})\subset \mathcal W_{\nu} \times \mathcal W^{\nu}_{\mu}$. Let $S$ be a $\mathbb{C}$-scheme, given an $S$-point $x$ in $X$, we can write $\mathbf{p}(x)=g\cdot y$ for a unique pair $(g,y)\in \mathcal W_{\nu}(S)\times \mathcal U^{\nu}(S)$, so $g^{-1}\cdot x\in \mathbf{p}^{-1}(\mathcal U^{\nu})(S)$, and we can write $x=g\cdot(g^{-1}\cdot x)$, this gives rise to a morphism $f:X\to \mathcal W_{\nu} \times \mathbf{p}^{-1}(\mathcal U^{\nu})$. Conversely, the action map $\mathcal W_{\nu} \times \mathbf{p}^{-1}(\mathcal U^{\nu})\to \mathcal W_{\mu}$ factors through $X$, and it's easy to see that this is inverse to $f$.

Since Weyl conjugations of $ {\mathcal U}^{\nu}$ is a covering of $\Gr^{\nu}$, and same method can be applied to $\mathcal W_{\nu}\times w\mathcal U^{\nu}w^{-1}$ for Weyl group element $w$, hence we conclude that a transversal slice of $\mathbf{p}^{-1}(\mathrm{Gr}^{\nu})=\mathcal W^{\nu}_{\mu}$ in $\mathcal W_{\mu}$ is isomorphic to $\mathcal W_{\nu}$.
\end{proof}

One might wonder if this result, after truncation to finite $\lambda$, directly proves the Theorem \ref{Thm: Transversal Slice}. The answer is ``No", because the image of $\bW^{\lambda}_{\nu}\times \mathbf{p}^{-1}(\mathcal U^{\nu})$ under $m_l$ is not contained in $\bW^{\lambda}_{\mu}$. For example, consider $G=\SL_2$, $\mu=0,\nu=\alpha,\lambda=2\alpha$, where $\alpha$ is the unique positive coroot, in this case generalized transversal slices are usual transversal slices so we can embed them as subvarieties of $\Gr_G$. Then we take two elements
\begin{align*}
\begin{bmatrix}
1 & 0\\
z^{-3} & 1\\
\end{bmatrix}
\cdot z^{\alpha}\in \bW^{2\alpha}_{\alpha},\; 
\begin{bmatrix}
1 & z\\
0 & 1\\
\end{bmatrix}
\cdot z^{\alpha}\in \mW_0\cap \mathcal U^{\alpha}.
\end{align*}
The map $m_l$ sends this pair to 
\begin{align*}
\begin{bmatrix}
1 & 0\\
z^{-3} & 1\\
\end{bmatrix}
\cdot
\begin{bmatrix}
1 & z\\
0 & 1\\
\end{bmatrix}
\cdot z^{\alpha}=
\begin{bmatrix}
z & 1\\
z^{-2} & z^{-1}+z^{-3}\\
\end{bmatrix},
\end{align*}
which is not in $\bGr^{2\alpha}$. So we have to do something more to go from infinite $\lambda$ to finite $\lambda$.

\subsection{Finite $\boldsymbol {\lambda}$}

First of all, we have shown above that the map $$m_r:\mathcal U^{\nu}\times \mathcal W_{\nu}\to \mathrm{Gr},$$is an open immersion such that $\mathcal U^{\nu}\times 1$ (resp. $1\times \mathcal W_{\nu}$) are mapped identically to $\mathcal U^{\nu}$ (resp. $\mathcal W_{\nu}$). Next lemma show that by considering $m_r$ instead of $m_l$, the truncation to finite $\lambda$ gives rise to a transversal slice of $\mathcal U^{\nu}$ in $\bGr^{\lambda}$.

\begin{lemma}
By identifying $\mathcal U^{\nu}\times \mathcal W_{\nu}$ with the image of $m_r$, we have $$(\mathcal U^{\nu}\times \mathcal W_{\nu})\cap \overline{\mathrm {Gr}}^{\lambda}=\mathcal U^{\nu}\times \overline{\mathcal W}^{\lambda}_{\nu}.$$Hence the restriction of $m_r$ to $\mathcal U^{\nu}\times \overline{\mathcal W}^{\lambda}_{\nu}$ gives an open immersion to $\overline{\mathrm {Gr}}^{\lambda}$.
\end{lemma}

\begin{proof}
Let $S$ be a $\mathbb{C}$-scheme, given an $S$-point $x$ in $(\mathcal U^{\nu}\times \mathcal W_{\nu})\cap \overline{\mathrm {Gr}}^{\lambda}$, we can write $x=g\cdot h$ for a unique pair $(g,h)\in \mathcal U^{\nu}(S)\times \mathcal W_{\nu}(S)$. Since $\mathcal U^{\nu}\subset G[z]$ and $\overline{\mathrm {Gr}}^{\lambda}$ is stable under the action of $G[z]$, we see that $h=g^{-1}x\in \overline{\mathcal W}^{\lambda}_{\nu}(S)$. Conversely, for each pair $(g,h)\in \mathcal U^{\nu}(S)\times \mathcal W_{\nu}(S)$, we have $gh\in (\mathcal U^{\nu}\times \mathcal W_{\nu})\cap \overline{\mathrm {Gr}}^{\lambda}(S)$. We establish an equality for $S$-points, whence the lemma follows.
\end{proof}

Next we lift these results to $G(\!(z^{-1})\!)$. Denote by $\widetilde{\mathcal U}^{\nu}$ the ind-scheme $\mathcal U^{\nu}z^{\nu}G[z]$ which is open in $G[z]z^{\mu}G[z]$, here we also identify $\mathcal U^{\nu}$ with subgroup of $ z^{\nu} G_1[\![z^{-1}]\!]z^{-\nu}$. Recall that $G[z]z^{\lambda}G[z]$ (resp. $\overline{G[z]z^{\lambda}G[z]}$) is denoted by $\mathcal {X}^{\lambda}$ (resp. $\overline{\mathcal {X}}^{\lambda}$) in \cite{muthiah2019symplectic}. Then we have open immersion $$m_r:\widetilde{\mathcal U}^{\nu}\times \overline{\mathcal W}^{\lambda}_{\nu}\hookrightarrow \overline{\mathcal {X}}^{\lambda},\; (g_1z^{\nu}g_2,h)\mapsto g_1hz^{\nu}g_2$$To state the next result, let's slightly generalize $\mathcal W_{\mu}$ and define $$\prescript{\lambda_1}{\lambda_2}{\mathcal{X}}_{\mu}=z^{\lambda_1}U_1[\![z^{-1}]\!]z^{-\lambda_1}\cdot z^{\mu} \cdot z^{-\lambda_2}B^{-}_1[\![z^{-1}]\!]z^{\lambda_2},$$ for dominant coweights $\lambda_1,\lambda_2$ and a general coweight $\mu$. Note that $\prescript{\lambda_1}{\lambda_2}{\mathcal{X}}_{\mu}$ is a scheme and in fact it's a closed subscheme of the ind-scheme ${\mathcal{X}}_{\mu}$ defined in \cite{muthiah2019symplectic}. The same argument in \textit{loc. cit.} shows that

\begin{lemma}\label{Lemma: Isomorphism}
The scheme $\prescript{\lambda_1}{\lambda_2}{\mathcal{X}}_{\mu}$ is isomorphic to $\mathcal U^{\lambda_1}\times \mathcal W_{\mu}\times \mathcal U^{\lambda_2^*}$. The scheme $\prescript{\lambda_1}{\lambda_2}{\mathcal{X}}_{\mu}\cap \overline{\mathcal X}^{\lambda}$ is isomorphic to $\mathcal U^{\lambda_1}\times \overline{\mathcal W}^{\lambda}_{\mu}\times \mathcal U^{\lambda_2^*}$.
\end{lemma}

\begin{remark}
Here we identify $z^{\lambda_1}U_1[\![z^{-1}]\!]z^{-\lambda_1}\cap G[z]$ (resp. $z^{-\lambda_2}U^{-}_1[\![z^{-1}]\!]z^{\lambda_2}\cap G[z]$) with $\mathcal U^{\lambda_1}$ (resp. $\mathcal U^{\lambda_2^*}$).
\end{remark}

\begin{lemma}\label{Lemma: Invariance}
For every dominant coweight $\nu$, the left multiplication of $ z^{\nu} G_1[\![z^{-1}]\!]z^{-\nu}$ stabilizes $\prescript{\nu}{\lambda}{\mathcal{X}}_{\mu}$ for sufficiently dominant $\lambda$.
\end{lemma}

\begin{proof}
By a conjugation of $z^{-\nu}$, it's enough to prove the lemma for $\nu=0$. Then we have 
\begin{align*}
G_1[\![z^{-1}]\!]\prescript{0}{\lambda}{\mathcal{X}}_{\mu}&=G_1[\![z^{-1}]\!]z^{\mu}\cdot z^{-\lambda}B^{-}_1[\![z^{-1}]\!]z^{\lambda}\\
&=U_1[\![z^{-1}]\!]z^{\mu}\cdot T_1[\![z^{-1}]\!] \cdot z^{-\mu}U^{-}_1[\![z^{-1}]\!]z^{\mu}\cdot z^{-\lambda}B^{-}_1[\![z^{-1}]\!]z^{\lambda}
\end{align*}
Then for sufficiently dominant $\lambda$, $z^{-\mu}U^{-}_1[\![z^{-1}]\!]z^{\mu}$ is a subgroup of $z^{-\lambda}B^{-}_1[\![z^{-1}]\!]z^{\lambda}$, whence the lemma follows.
\end{proof}

\begin{lemma}
Identify $\widetilde{\mathcal U}^{\nu}\times \overline{\mathcal W}^{\lambda}_{\nu}$ with the image of $m_r$ in $\overline{\mathcal {X}}^{\lambda}$, then for sufficiently dominant $\lambda'$, we have $$m_r\left( (\prescript{\nu}{\lambda'}{\mathcal{X}}_{\mu}\cap \widetilde{\mathcal U}^{\nu})\times \overline{\mathcal W}^{\lambda}_{\nu}\right)\subset \prescript{\nu}{\lambda'}{\mathcal{X}}_{\mu}\cap \overline{\mathcal {X}}^{\lambda}.$$
\end{lemma}

\begin{proof}
Let's take $g_1z^{\nu}g_2\in \prescript{\nu}{\lambda'}{\mathcal{X}}_{\mu}\cap \widetilde{\mathcal U}^{\nu}(S)$ for an $\mathbb C$-scheme $S$, and for arbitrary $h\in \overline{\mathcal W}^{\lambda}_{\nu}(S)$, we have $$m_r(g_1z^{\nu}g_2,h)=g_1hz^{\nu}g_2=g_1hg_1^{-1}\cdot g_1z^{\nu}g_2$$and since both $g_1$ and $h$ are elements in $z^{\nu} G_1[\![z^{-1}]\!]z^{-\nu}(S)$, we see that $m_r(g_1z^{\nu}g_2,h)$ is in the $z^{\nu} G_1[\![z^{-1}]\!]z^{-\nu}$ orbit of $g_1hz^{\nu}g_2$, hence it lies inside $\prescript{\nu}{\lambda'}{\mathcal{X}}_{\mu}(S)$ for sufficiently dominant $\lambda'$ by Lemma \ref{Lemma: Invariance}.
\end{proof}

Next, we embed $(\mW_{\mu}\cap \widetilde{\mathcal{U}}^{\nu})\times \overline{\mathcal W}^{\lambda}_{\nu}$ into $(\prescript{\nu}{\lambda'}{\mathcal{X}}_{\mu}\cap \widetilde{\mathcal U}^{\nu})\times \overline{\mathcal W}^{\lambda}_{\nu}$ using Lemma \ref{Lemma: Isomorphism},  compose it with $m_r$ so that the image is in $\prescript{\nu}{\lambda'}{\mathcal{X}}_{\mu}\cap \overline{\mathcal {X}}^{\lambda}$, and then compose with the projection to $\bW^{\lambda}_{\mu}$ using Lemma \ref{Lemma: Isomorphism} again. Denote this map by $\psi$, and we have the following

\begin{lemma}\label{Lemma: Psi is mono}
$\psi:(\mW_{\mu}\cap \widetilde{\mathcal{U}}^{\nu})\times \overline{\mathcal W}^{\lambda}_{\nu}\to \bW^{\lambda}_{\mu}$ is a monomorphism of schemes.
\end{lemma}

\begin{proof}
We need to show that $\psi$ is injective on $S$-points for all $\mathbb C$-scheme $S$. Let's take $g_1z^{\nu}g_2\in \mW_{\mu}\cap \widetilde{\mathcal U}^{\nu}(S)$ and $h\in \overline{\mathcal W}^{\lambda}_{\nu}(S)$, where we use the decomposition $\widetilde{\mathcal U}^{\nu}={\mathcal U}^{\nu}z^{\nu}G[z]$, in other word $g_1\in \mathcal U^{\nu}(S)$ and $g_2\in G[z](S)$, then $$\psi(g_1z^{\nu}g_2,h)=g_1hg_1^{-1}\cdot g_1z^{\nu}g_2 \mod \mathcal U^{\nu}\times \mathcal U^{\lambda'^*}$$ where $\mathcal U^{\nu}$ acts from the left and $\mathcal U^{\lambda'^*}$ acts from the right. Now $g_1hg_1^{-1}\in z^{\nu} G_1[\![z^{-1}]\!]z^{-\nu}(S)=\mathcal U^{\nu}(S)\times \mathcal W_{\nu}(S)$, so there exists a unique pair $(s,t)\in \mathcal U^{\nu}(S)\times \mathcal W_{\nu}(S)$ such that $g_1hg_1^{-1}=st$. As a result, we have $$\psi(g_1z^{\nu}g_2,h)=tg_1z^{\nu}g_2 \mod \mathcal U^{\lambda'^*}.$$Suppose that there is another pair $(g_1'z^{\nu}g_2',h')\in \mW_{\mu}\cap \widetilde{\mathcal U}^{\nu}(S) \times \overline{\mathcal W}^{\lambda}_{\nu}(S)$ such that $\psi(g_1z^{\nu}g_2,h)=\psi(g_1'z^{\nu}g_2',h')$, then from above computation we see that $tg_1z^{\nu}g_2=t'g_1'z^{\nu}g_2' \mod \mathcal U^{\lambda'^*}$, thus $$tg_1z^{\nu}=t'g_1'z^{\nu} \mod G[z]$$Since $m_l:\mW_{\nu}\times \mathcal U^{\nu}\to \Gr$ is an open immersion, we conclude that $t=t'$ and $g_1=g_1'$.

From these two equations and the definition of $(s,t)$, we arrive at $g_1hg_1^{-1}=\tilde{s}g_1h'g_1^{-1}$, where $\tilde{s}=ss'^{-1}\in \mathcal U^{\nu}(S)$, and this implies that $$hh'^{-1}=g_1^{-1}\tilde{s}g_1\in \mathcal U^{\nu}(S),$$ but by definition $hh'^{-1}\in \mW_{\nu}(S)$, hence $h=h'$.

Finally, plug three equations we already have into the condition $\psi(g_1z^{\nu}g_2,h)=\psi(g_1'z^{\nu}g_2',h')$, and we get $g_1hz^{\nu}g_2=g_1hz^{\nu}g_2' \mod \mathcal U^{\lambda'^*}$, equivalently $g_2'=g_2y$ where $y\in \mathcal U^{\lambda'^*}(S)$. Then both $g_1z^{\nu}g_2$ and $g_1z^{\nu}g_2y$ are elements in $\mW_{\nu}(S)$, and this implies that $$y\in \mW_{\nu}(S)\cap \mathcal U^{\lambda'^*}(S)=1$$This concludes the proof of $(g_1z^{\nu}g_2,h)=(g_1'z^{\nu}g_2',h')$.
\end{proof}

\begin{proof}[Proof of Theorem \ref{Thm: Transversal Slice}]
By Lemma \ref{Lemma: Psi is mono}, $\psi:(\mW_{\mu}\cap \widetilde{\mathcal{U}}^{\nu})\times \overline{\mathcal W}^{\lambda}_{\nu}\to \bW^{\lambda}_{\mu}$ is a monomorphism between normal varieties of the same dimension, in particular $\psi$ induces isomorphism between generic points, thus $\psi$ is an open immersion by Zariski's main theorem \cite[\href{https://stacks.math.columbia.edu/tag/05K0}{Tag 05K0}]{stacks-project}. Moreover it induces identity map on $(\mW_{\mu}\cap \widetilde{\mathcal{U}}^{\nu})\times 1$, so $\overline{\mathcal W}^{\lambda}_{\nu}$ is a transversal slice of $\mW_{\mu}\cap \widetilde{\mathcal{U}}^{\nu}$ in $\bW^{\lambda}_{\mu}$. Since Weyl conjugations of $ \widetilde{\mathcal U}^{\nu}$ is a covering of $\mathcal{X}^{\nu}$, and same method can be applied to $\mW_{\mu}\cap w\widetilde{\mathcal U}^{\nu}w^{-1}$ for Weyl group element $w$, and see that there is an open immersion $\psi_w:(\mW_{\mu}\cap w\widetilde{\mathcal{U}}^{\nu}w^{-1})\times \overline{\mathcal W}^{\lambda}_{\nu}\to \bW^{\lambda}_{\mu}$ which induces identity map on $(\mW_{\mu}\cap w\widetilde{\mathcal{U}}^{\nu}w^{-1})\times 1$.
\end{proof}

\section{Symplectic Singularities}\label{Sec: Symplectic Singularities}

In this section we prove Conjecture \ref{Conj: Symplectic Singularities} for $\overline {\mathcal W}^{\lambda}_{\mu}$.
Note that our main result (Theorem \ref{Thm: Symplectic Singularities}) can be deduced using Coulomb branch techniques, see \cite{weekes2020quiver} for ADE-type and \cite{nakajima2019coulomb} for extension to non-simply laced cases. Here we provide an elementary approach.

\begin{corollary}\label{Cor: Rational Gorenstein}
$\bW^{{\lambda}}_{\mu}$ has rational Gorenstein singularities.
\end{corollary}

\begin{proof}
This is known for dominant $\mu$ \cite{kamnitzer2014yangians}. The general case follows from the dominant case and the transversal slice description of local singularities of $\bW^{{\lambda}}_{\mu}$ (Theorem \ref{Thm: Transversal Slice}).
\end{proof}


Before stating more results, let's recall a theorem of Namikawa \cite[Theorem 6]{namikawa2000extension}, which is essential in the proof of our main result.

\begin{theorem}[Namikawa]\label{Nam1}
A normal variety $X$ has symplectic singularities if and only if there is a symplectic form $\Omega$ on the smooth locus $X^{\mathrm{reg}}$, and $X$ has rational Gorenstein singularities.
\end{theorem}

This means that we only need to show that there exists a natural Poisson structure on $\overline{\mathcal{W}}^{\lambda}_{\mu}$ and it's regular on ${\mathcal{W}}^{\lambda}_{\mu}$. Recall that there is an open embedding $\mathring{Z}^{\alpha^*}=\overline{\mathcal{W}}^{0}_{\mu-\lambda}\hookrightarrow \overline{\mathcal{W}}^{\lambda}_{\mu}$, and in fact this is a section on $ \mathring{Z}^{\alpha^*}$ of the natural map $\mathbf{q}:\overline{\mathcal{W}}^{\lambda}_{\mu}\to Z^{\alpha^*}$ \cite{braverman2016coulomb}.

\begin{proposition}\label{Prop: Poisson Extension}
The natural Poisson structure on $\mathring{Z}^{\alpha^*}$ extends to a Poisson structure on $\overline{\mathcal{W}}^{\lambda}_{\mu}$, and it's regular on ${\mathcal{W}}^{\lambda}_{\mu}$.
\end{proposition}

For simply laced $G$, this is \cite[Proposition 3.18]{braverman2016coulomb}. Here we present a proof for general simple $G$. From now on, we will freely use notations from  \cite{braverman2016coulomb}. Recall that there is an smooth open subvariety $\dt{Z}^{\alpha^*}$ of ${Z}^{\alpha^*}$, such that $\mathbf{q}^{-1}({Z}^{\alpha^*}\setminus\dt{Z}^{\alpha^*})$ has codimension $2$ in $\overline{\mathcal{W}}^{\lambda}_{\mu}$, so it's enough to prove the proposition for $\dt{\mathcal{W}}^{\lambda}_{\mu}=\mathbf{q}^{-1}(\dt{Z}^{\alpha^*})$.

The following is a generalization of \cite[Lemma 2.11, 2.12]{braverman2016coulomb} to non-simply laced case, and its proof is identical to that of \textit{loc. cit.}.
\begin{lemma}
Let $I$ be the set of simple coroots, and let $d_i=1$ for short coroots and $d_i=2$ or $3$ for long coroots, then:
\begin{itemize}
    \item[(1)] $\mathrm{div}\left( F_{\alpha^*}\right)=\sum_{i\in I}d_i\partial_i\dt{Z}^{\alpha^*}$;
    \item[(2)] $\mathrm{div} \left(\mathbf{q}^*F_{\alpha^*}\right)=\sum_{i\in I}d_i\mathbf{q}^{-1}_*\left(\partial_i\dt{Z}^{\alpha^*}\right)+\sum_{i\in I}\langle \lambda,d_i\check{\alpha}_{i^*}\rangle \dt{E}_i$;
    \item[(3)] $\mathbf{p}^* \mathcal L=\mathcal{O}_{\dt{\mathcal {GZ}}^{-\mu}_{\lambda}}\left(\sum_{i\in I}d_i\mathbf{q}^*\left(\partial_i\dt{Z}^{\alpha^*}\right)-\sum_{i\in I}\langle \lambda,d_i\check{\alpha}_{i^*}\rangle \dt{E}_i\right)\cong \mathcal{O}_{\dt{\mathcal {GZ}}^{-\mu}_{\lambda}}\left(-\sum_{i\in I}\langle \lambda,d_i\check{\alpha}_{i^*}\rangle \dt{E}_i\right)$;
    \item[(4)] $\mathrm{div}\left(\mathbf{q}^*\pi_{\alpha^*}^*\mathbb{A}^{\alpha^*}_i\right)=E_i+\mathbf{q}^{-1}_*\left(\pi_{\alpha^*}^*\mathbb{A}^{\alpha^*}_i\right)$.
\end{itemize}
\end{lemma}

Define the ideal sheaf $\mathcal K_i=\mathcal{I}_i^{\langle \lambda,d_i\check{\alpha}_{i^*}\rangle}+\mathcal J_i^{d_i}$, where $\mathcal{I}_i$ (resp. $\mathcal J_i$) is the ideal sheaf of divisor $\partial_i\dt{Z}^{\alpha^*}$ (resp. $\pi_{\alpha^*}^*\mathbb{A}^{\alpha^*}_i$). Note that subvarieties defined by different $\mathcal K_i$ are disjoint. Then define $\mathcal K=\cap_{i\in I}\mathcal K_i$. The same argument in the proof of \cite[Proposition 2.10]{braverman2016coulomb} shows that $\mathbf{q}^{-1}\mathcal K\cdot\mathcal O_{\dt{\mathcal {GZ}}^{-\mu}_{\lambda}}\cong \mathbf{p}^*\mathcal L$.

The universal property of blow up implies that there is projective morphism $\Upsilon : \dt{\mathcal {GZ}}^{-\mu}_{\lambda}\to \mathrm{Bl}_{\mathcal K}\dt{Z}^{\alpha^*}$. Moreover, the fiber of $\mathbf{q}:\dt{\mathcal {GZ}}^{-\mu}_{\lambda}\to \dt{Z}^{\alpha^*}$ is either a point or isomorphic to $\mathbb{P}^1$, and the fiber of $\mathrm{Bl}_{\mathcal K}\dt{Z}^{\alpha^*}\to \dt{Z}^{\alpha^*}$ has dimension $\le 1$ (since locally $\mathcal K$ is generated by $2$ elements). It follows that $\Upsilon$ is finite, hence it's a normalization map. In other words we have shown the following generalization of \cite[Proposition 2.10]{braverman2016coulomb}

\begin{proposition}
The identity isomorphism over $\pi_{\alpha^*}^{-1}(\mathbb{G}_m^{\alpha^*}) $ extends to a regular isomorphism of varieties over $ \dt{Z}^{\alpha^*}$: $$ \dt{\mathcal {GZ}}^{-\mu}_{\lambda}\overset{\sim}{\longrightarrow} \widetilde{\mathrm{Bl}}_{\mathcal K}\dt{Z}^{\alpha^*},\; \dt{\mathcal {W}}^{\lambda}_{\mu}\overset{\sim}{\longrightarrow} \widetilde{\mathrm{Bl}}^{\mathrm{aff}}_{\mathcal K}\dt{Z}^{\alpha^*}.$$
\end{proposition}

Replacing ${\mathrm{Bl}}^{\mathrm{aff}}_{\mathcal K}$ by $\widetilde{\mathrm{Bl}}^{\mathrm{aff}}_{\mathcal K}$ in \cite[(2.13)]{braverman2016coulomb}, we see that \cite[Proposition 2.9]{braverman2016coulomb} is true for all simple $G$, without assumption of ADE-type (this is pointed out in Remark 2.14 of \textit{loc. cit.}). We record it here for completeness.

\begin{proposition}[Factorization]\label{Prop: Factorization}
The birational isomorphism $\varphi$ extends to a regular isomorphism of varieties over $\left(\mathbb{G}_m^{\beta^*}\times \mathbb A^1\right)_{\mathrm{disj}}$: $$\left(\mathbb{G}_m^{\beta^*}\times \mathbb A^1\right)_{\mathrm{disj}}\times_{\mathbb A^{\alpha^*}}\overline{\mathcal{W}}^{\lambda}_{\mu}\overset{\sim}{\longrightarrow} \left(\mathbb{G}_m^{\beta^*}\times \mathbb A^1\right)_{\mathrm{disj}}\times_{\mathbb A^{\beta^*}\times \mathbb A^1} \left(\mathring{Z}^{\beta^*}\times \mathcal S_{\langle \lambda,\check{\alpha}_{i^*}\rangle} \right)$$where $\beta=\alpha-\alpha_i$, and $\mathcal S_{\langle \lambda,\check{\alpha}_{i^*}\rangle} $ is the $A_{\langle \lambda,\check{\alpha}_{i^*}\rangle-1}$ singularity defined by equation $xy=w^{\langle \lambda,\check{\alpha}_{i^*}\rangle}$, and the natural map $\mathcal S_{\langle \lambda,\check{\alpha}_{i^*}\rangle} \to Z^{\alpha_i^*}\cong \mathbb A^2$ is given by $(x,y,w)\mapsto (y,w)$.
\end{proposition}

\begin{proof}[Proof of Proposition \ref{Prop: Poisson Extension}]
For type $A_1$, this is a straightforward computation. For general $G$, we deduce from Proposition \ref{Prop: Factorization} that the Poisson structure extends over $\dt{E}_i$ for every $i\in I$, so it extends globally. And it's regular outside of a set of codimension $2$, hence it's regular on $\mathcal W^{\lambda}_{\mu}$.
\end{proof}

Combine Corollary \ref{Cor: Rational Gorenstein}, Theorem \ref{Nam1} and Proposition \ref{Prop: Poisson Extension}, we obtain the main results of this section:
\begin{theorem}\label{Thm: Symplectic Singularities}
$\overline{\mathcal{W}}^{\lambda}_{\mu}$ has symplectic singularities.
\end{theorem}

\subsection{$\mathbb Q$-Factorial Terminalization}
Assume that $G$ is of adjoint type, and suppose that $\lambda=\sum_{i=1}^n \lambda_i$, where $\lambda_i$ are fundamental coweights, then consider the convolution
$$m:\mathrm{Gr}^{\vec{\lambda}}=\mathrm{Gr}^{\lambda_1}\widetilde{\times}\cdots \widetilde{\times}\mathrm{Gr}^{\lambda_n}\subset \overline{\mathrm{Gr}}^{\vec{\lambda}}=\overline{\mathrm{Gr}}^{\lambda_1}\widetilde{\times}\cdots \widetilde{\times}\overline{\mathrm{Gr}}^{\lambda_n}\to \overline{\mathrm{Gr}}^{\lambda} $$and define $$\mW^{\vec{\lambda}}_{\mu}=\mathrm{Gr}^{\vec{\lambda}}\times _{\overline{\mathrm{Gr}}^{\lambda}}\bW^{\lambda}_{\mu},\; \bW^{\vec{\lambda}}_{\mu}=\overline{\mathrm{Gr}}^{\vec{\lambda}}\times _{\overline{\mathrm{Gr}}^{\lambda}}\bW^{\lambda}_{\mu}$$
Here what we called $\bW^{\vec{\lambda}}_{\mu}$ is denoted by $\widetilde{\mathcal W}^{\underline{\lambda}}_{\mu}$ in \cite{braverman2018line}. Note that $m:\bW^{\vec{\lambda}}_{\mu}\to \bW^{\underline{\lambda}}_{\mu}$ is projective and birational, in fact it's an isomorphism on the open loucs $\mathring{Z}^{\alpha^*}$, so $\bW^{\vec{\lambda}}_{\mu}$ inherits a Poisson structure on $\mathring{Z}^{\alpha^*}$. The same commutative diagram in the proof of Proposition \ref{Prop: Smoothness} with $\mathsf H^{\underline{\lambda}}_{\mathrm{BD}}$ replaced by a convolution version $\mathsf H^{\vec{\lambda}}_{\mathrm{BD}}$ shows the following:
\begin{proposition}\label{Prop: Smooth Locus 2}
$\bW^{\vec{\lambda}}_{\mu}$ is normal and its smooth locus is $\mW^{\vec{\lambda}}_{\mu}$.
\end{proposition}

The factorization of $\bW^{\vec{\lambda}}_{\mu}$ is also proven in \cite{braverman2018line}, and it holds for general simple $G$ without assumption on ADE type. Here we record it for completeness:

\begin{proposition}[Factorization]\label{Prop: Factorization2}
The birational isomorphism $\varphi$ extends to a regular isomorphism of varieties over $\left(\mathbb{G}_m^{\beta^*}\times \mathbb A^1\right)_{\mathrm{disj}}$: $$\left(\mathbb{G}_m^{\beta^*}\times \mathbb A^1\right)_{\mathrm{disj}}\times_{\mathbb A^{\alpha^*}}\overline{\mathcal{W}}^{\vec{\lambda}}_{\mu}\overset{\sim}{\longrightarrow} \left(\mathbb{G}_m^{\beta^*}\times \mathbb A^1\right)_{\mathrm{disj}}\times_{\mathbb A^{\beta^*}\times \mathbb A^1} \left(\mathring{Z}^{\beta^*}\times \widetilde{\mathcal S}_{\langle \lambda,\check{\alpha}_{i^*}\rangle} \right)$$where $\beta=\alpha-\alpha_i$, and $\widetilde{\mathcal S}_{\langle \lambda,\check{\alpha}_{i^*}\rangle} $ is the minimal resolution of the $A_{\langle \lambda,\check{\alpha}_{i^*}\rangle-1}$ singularity $\mathcal S_{\langle \lambda,\check{\alpha}_{i^*}\rangle}$.
\end{proposition}

\begin{theorem}\label{Thm: Terminalization}
$m:\overline{\mathcal{W}}^{\vec{\lambda}}_{\mu}\to \overline{\mathcal{W}}^{\lambda}_{\mu}$ is a $\mathbb Q$-factorial terminalization map.
\end{theorem}

\begin{proof}
We observe that
\begin{itemize}
    \item[(1)] $\forall i,j\in I$, $\lambda_i-\alpha_j$ is not dominant, thus every stratum ${\mathcal{W}}^{\vec{\lambda}'}_{\mu}$ in the singular locus has codimension at least $4$, i.e. $\overline{\mathcal{W}}^{\vec{\lambda}}_{\mu}$ is smooth at codimension $3$;
    \item[(2)] The pull back of the canonical Poisson structure along the resolution map $\widetilde{\mathcal S}_{\langle \lambda,\check{\alpha}_{i^*}\rangle}\to\mathcal S_{\langle \lambda,\check{\alpha}_{i^*}\rangle}$ is a symplectic structure, so by the factorization we see that the Poisson structure on the open locus $\mathring{Z}^{\alpha^*}\subset \overline{\mathcal{W}}^{\vec{\lambda}}_{\mu}$ extends over the exceptional divisors $E_i$ for all $i\in I$, moreover it's symplectic outside of a subset of codimension $2$;
    \item[(3)] $\overline{\mathcal{W}}^{\vec{\lambda}}_{\mu}\to \Bun_B^{w_0\mu}(\mathbb P^1,\infty)$ is an open substack of a locally trivial fibration with fibers isomorphic to $\bGr^{\vec{\lambda}^*}$, and the latter is known to be $\mathbb Q$-factorial \cite[Theorem 2.7]{kamnitzer2014yangians}, thus $\overline{\mathcal{W}}^{\vec{\lambda}}_{\mu}$ is $\mathbb Q$-factorial as well.
\end{itemize}
The first two facts together with Flenner's extendability theorem \cite{flenner1988extendability} shows that $\overline{\mathcal{W}}^{\vec{\lambda}}_{\mu}$ has symplectic singularities. Moreover Namikawa shows that regularity at codimension $3$ implies that it's terminal \cite{namikawa2001note}.

Finally we only need to show that $m:\overline{\mathcal{W}}^{\vec{\lambda}}_{\mu}\to \overline{\mathcal{W}}^{\lambda}_{\mu}$ has zero discrepancy. Again, this follows from the factorization. Note that the minimal resolution $\widetilde{\mathcal S}_{\langle \lambda,\check{\alpha}_{i^*}\rangle}\to\mathcal S_{\langle \lambda,\check{\alpha}_{i^*}\rangle}$ is well-known to have zero discrepancy.
\end{proof}

We also have a partial generalization of \cite[Theorem 2.9]{kamnitzer2014yangians}.
\begin{proposition}
The following are equivalent:
\begin{itemize}
\item[(1)] $\bW^{\lambda}_{\mu}$ has a symplectic resolution.
\item[(2)] $\bW^{\vec{\lambda}}_{\mu}$ is smooth.
\item[(3)] $\mW^{\vec{\lambda}}_{\mu}=\bW^{\vec{\lambda}}_{\mu}$.
\item[(4)] If there are dominant coweights $\nu_i,i=1,\cdots, n$, such that $\nu_i\le \lambda_i$ and $\mu\le \sum_{i=1}^n \nu_i$, then we must have $\nu_i=\lambda_i$.
\end{itemize}
\end{proposition}

\begin{proof}
$(2)\Longrightarrow (1)$ is tautological. $(2)\Longleftrightarrow (3)$: $\bW^{\vec{\lambda}}_{\mu}$ is smooth if and only if it agrees with its smooth locus, which is equivalent to $\mW^{\vec{\lambda}}_{\mu}=\bW^{\vec{\lambda}}_{\mu}$ by Proposition \ref{Prop: Smooth Locus 2}.

$(3)\Longleftrightarrow (4)$: Note that $$\bW^{\vec{\lambda}}_{\mu}= \mW^{\vec{\lambda}}_{\mu}\cup \coprod_{\vec{\nu}}\mW^{\vec{\nu}}_{\mu}$$where $\vec{\nu}$ satisfy condition $(3)$, then the equivalence is obvious.

$(1)\Longrightarrow (2)$: For dominant $\mu$, this is already proven in \cite[Theorem 2.9]{kamnitzer2014yangians}. For general $\mu$, there is a subtle complication: \textit{loc. cit.} uses \cite[5.6]{namikawa2011poisson} where a conical $\mathbb G_m$ action is assumed, but the natural $\mathbb G_m$ action on $\bW^{\lambda}_{\mu}$ is not conical \footnote{The author thanks Vasily Krylov for pointing out this subtlety.}.

Assume that $(2)$ is false, then by what we have shown, $(4)$ is false, i.e. there exists dominant coweights $\nu_i\le \lambda_i$ such that $\mu\le \sum_i \nu_i$ and there is some $i$ that $\nu_i\neq \lambda_i$. Define $\nu= \sum_i \nu_i$, then the implication $(1)\Longrightarrow (2)$ for dominant $\nu$ implies that $\overline{\mathcal W}^{\lambda}_{\nu}$ has no symplectic resolution. We shall use this to show that in this case, (1) must be false.

Note that for symplectic singularity $Y$, a symplectic resolution $f:X\to Y$ is the same as a crepant resolution of $Y$, since nondegeneracy of $f^*\omega_Y$ is equivalent to $f^*\omega_Y^n$ having no zero along exceptional divisors, which in turn is equivalent to $f$ being crepant.

Suppose that (1) holds, then the restriction of a symplectic resolution to an open locus gives rise to a crepant resolution of $\overline{\mathcal W}^{\lambda}_{\nu}\times \mathcal W^{\nu}_{\mu}$, by Theorem \ref{Thm: Transversal Slice}. Then we apply the following lemma to the crepant resolution of $\overline{\mathcal W}^{\lambda}_{\nu}\times \mathcal W^{\nu}_{\mu}$, we get a crepant resolution of $\overline{\mathcal W}^{\lambda}_{\nu}$, which is automatically symplectic by previous discussion. However this contradicts with what we have already deduced, so we conclude that $(1)$ must be false.
\end{proof}

\begin{lemma}
Suppose that $X$ is a normal variety, and $Y$ is a smooth variety, and $f: Z\to X\times Y$ is a crepant resolution. Then there exists a Zariski open subset $U$ in $Y$ such that $\forall y\in U$, the fiber $f_y: Z_y\to X\times {y}$ is a crepant resolution.
\end{lemma}

\begin{proof}
By generic smoothness theorem, there exists Zariski open subset $U$ in $Y$ such that the composition of $f$ and projection to $Y$ is smooth. Shrink $U$ if needed, we can assume that $\forall y\in U$, $f_y$ is birational and is isomorphism on the smooth locus of $X$, thus $f_y: Z_y\to X\times {y}$ is a resolution of singularity. Shrink $U$ again if needed, we assume that the canonical divisor $K_U$ is trivial, so the Weil divisor $[K_{X\times U}]$ can be written as $[K_X]\times U$. Now by the assumption that $f$ is crepant, we have $f^*([K_X]\times U)=[K_{Z_U}]$. Restrict the equation to fiber, and we see that $f_y^*([K_X])=[K_{Z_y}]$, $\forall y\in U$. This means that $f_y$ is a crepant resolution.
\end{proof}

\begin{example}
In Type $A$, fundamental coweights are minuscule, so condition (4) is automatically satisfied, thus $m:\bW^{\vec{\lambda}}_{\mu}\to \bW^{\lambda}_{\mu}$ is a symplectic resolution.
\end{example}

\bibliographystyle{alpha}
\bibliography{Bib}

\end{document}